\documentclass[12pt]{amsart}

\usepackage{amssymb}

\usepackage{enumitem}

\usepackage{graphicx}

\usepackage{hyperref}

\usepackage{times}
\usepackage{textcomp}
\makeatletter
\@namedef{subjclassname@2010}{%
  \textup{2010} Mathematics Subject Classification}
\makeatother

\usepackage[T1]{fontenc}
\newtheorem{theorem}{Theorem}[section]
\newtheorem{corollary}[theorem]{Corollary}
\newtheorem{lemma}[theorem]{Lemma}

\newtheorem{remark}[theorem]{Remark}
\theoremstyle{definition}
\newtheorem{definition}[theorem]{Definition}

\numberwithin{equation}{section}

\begin{document}

%%%%% To ease editing, for IMPAN journals add:

\baselineskip=17pt

%%%%%%%%%%%%%%%%

\title[post-critically finite polynomial of the form $x^d+c$]{On the orbit of a post-critically finite polynomial of the form $x^d+c$}

\author[V. Goksel]{Vefa Goksel}
\address{Mathematics Department\\ University of Wisconsin\\
Madison\\
WI 53706, USA}
\email{goksel@math.wisc.edu}

\date{}

\begin{abstract}
	In this paper, we study the critical orbit of a post-critically finite polynomial of the form $f_{c,d}(x) = x^d+c \in \mathbb{C}[x]$. We discover that in many cases the orbit elements satisfy some strong arithmetic properties. It is well known that the $c$ values for which $f_{c,d}$ has tail size $m\geq 1$ and period $n$ are the roots of a polynomial $G_d(m,n) \in \mathbb{Z}[x]$, and the irreducibility or not of $G_d(m,n)$ has been a great mystery. As a consequence of our work, for any prime $d$, we establish the irreducibility of these $G_d(m,n)$ polynomials for infinitely many pairs $(m,n)$. These appear to be the first known such infinite families of $(m,n)$. We also prove that all the iterates of $f_{c,d}$ are irreducible over $\mathbb{Q}(c)$ if $d$ is a prime and $f_{c,d}$ has a fixed point in its post-critical orbit.

\end{abstract}

\subjclass[2010]{Primary 11R09, 37P15}

\keywords{Misiurewicz point, post-critically finite, rigid divisibility sequence}

\maketitle

\section{Introduction}
Let $f(x)\in \mathbb{C}[x]$ be a polynomial of degree at least $2$. We denote by $f^n(x)$ the $n$th iterate of $f(x)$. Given $a\in \mathbb{C}$, one fundamental object in dynamics is the orbit $$O_a(f) = \{f(a),f^2(a),\dots\}$$ of $a$ under $f$. When this orbit is finite for all critical points of $f$, we call $f$ \textit{post-critically finite} (PCF). In this paper, we study a special case, namely the PCF polynomials of the form $f_{c,d}(x) = x^d + c\in \mathbb{C}[x]$ for $d\geq 2$.\\

$0$ is the unique critical point of $f_{c,d}$. Suppose $f_{c,d}$ is PCF, i.e., there exist $m,n\in \mathbb{Z}$ with $n\neq 0$ such that $f_{c,d}^m(0) = f_{c,d}^{m+n}(0)$. We say \textit{$f_{c,d}$ has exact type $(m,n)$} if $n$ is the minimal positive integer such that $f_{c,d}^m(0) = f_{c,d}^{m+n}(0)$ and $f_{c,d}^k(0) \neq f_{c,d}^{k+n}(0)$ for any $k < m$. When $m\geq 1$, a number $c_0$ for which $f_{c_0,d}$ has exact type $(m,n)$ is called a \textit{Misiurewicz point with period (m,n)}. It is known that Misiurewicz points with period $(m,n)$ are the roots of a monic polynomial $G_d(m,n)\in \mathbb{Z}[x]$. In particular, $c_0$ is always an algebraic integer. (\cite{Hutz}, Corollary $3.4$). To explain how the polynomial $G_d(m,n)$ is defined, we will follow the notation in (\cite{Hutz}): Let 
$$\Phi^*_{f,n}(x) = \prod_{k|n} (f^k(x)-x)^{\mu(n/k)}$$ be the standard dynatomic polynomial, and define the generalized dynatomic polynomial $\Phi^*_{f,m,n}$ by
$$\Phi^*_{f,m,n}(x) = \frac{\Phi^*_{f,n}(f^m(x))}{\Phi^*_{f,n}(f^{m-1}(x))}.$$ Then, the polynomial $G_d(m,n)$ is defined by 
$$
G_d(m,n) = \left\{
\begin{array}{ll}
\frac{\Phi^*_{f,m,n}(0)}{\Phi^*_{f,0,n}(0)^{d-1}} & \text{if } m\neq 0 \text{ and } n  |  (m-1) \\
\Phi^*_{f,m,n}(0) & \text{otherwise}.
\end{array}
\right.
$$

Having defined the polynomials $G_d(m,n)$, it is natural to ask: Which triples $(d,m,n)$ make $G_d(m,n)$ irreducible? This question is wide open. For a reference, see for instance \cite{Hutz} and \cite{Milnor}. There does not appear to exist any prior work which gives an infinite family of $(d,m,n)$ for which $G_d(m,n)$ is irreducible. In this direction, the following corollary to our main theorems gives the first known such infinite families of $(d,m,n)$.

\begin{corollary}
	$G_d(m,n)$ is irreducible in the following cases:
	\item[(i)] $m\neq 0$, $n=1$, $d$ is any prime.
	\item[(ii)] $m\neq 0$, $n=2$, $d=2$.
	
\end{corollary}
We also would like to mention that this corollary inspired a subsequent paper of Buff et al. (\cite{BEK}), where they extended this result by proving that for $k\geq 1$, both $G_{p^k}(m,1)$ and $G_{p^k}(m,2)$ have precisely $k$ different irreducible factors for $m\geq 1$ (\cite{BEK}, Theorem $3$ and Corollary $4$). They also proved that $G_2(m,3)$ is irreducible for $m\geq 1$ (\cite{BEK}, Corollary $5$).\\\\
Before giving the next corollary to our main theorems, we recall a definition from the theory of polynomial iteration: Let $F$ be a field, and $f(x)\in F[x]$ be a polynomial over $F$. We say that \textit{f is stable over F} if $f^n(x)$ is irreducible over $F$ for all $n\geq 1$.

\begin{corollary}
	Let $d$ be a prime, and suppose $f_{c,d}(x) = x^d+c$ has exact type $(m,1)$ for some $m\neq 0$. Set $K=\mathbb{Q}(c)$. Then, $f_{c,d}$ is stable over $K$.
\end{corollary}

Note that the simplest example of Corollary $1.2$ is the polynomial $f_{-2,2}(x) = x^2-2\in \mathbb{Q}[x]$, which is already well-known to be stable. Thus, Corollary $1.2$ can be thought of as a generalization of this well-known example.\\

We also would like to say a few words about why the stability question is harder when $f_{c,d}$ has exact type $(m,n)$ with $n>1$: By a result of Hamblen et al. (\cite{Hamblen}, Theorem $8$), proving the stability of $f_{c,d}$ comes down to show that there are no $\pm d$th powers in the critical orbit. However, when $n>1$, one of our main theorems implies that there always exist some unit elements in the critical orbit, and checking if these units are $\pm d$th powers or not appear to be a difficult problem.\\

Both Corollary $1.1$ and Corollary $1.2$ will follow from the next theorem, which establishes that the critical orbit elements for the PCF polynomials $f_{c,d}$ satisfy suprisingly strong properties when $d$ is a prime. We first fix the following notation, which we will also use throughout the paper:\\

Let $K$ be a number field, and $\mathcal{O}_K$ its ring of integers. Take $a\in \mathcal{O}_K$. Throughout, we will use $(a)$ to denote the ideal $a\mathcal{O}_K$. Also, for $f_{c,d}$ with exact type $(m,n)$, we will use $O_{f_{c,d}}=\{a_1,a_2,\dots,a_{m+n-1}\}$ to denote the critical orbit, where we set $a_i=f_{c,d}^i(0)$. Whenever we use $a_i$ for some $i>m+n-1$, we again obtain it by setting $a_i = f_{c,d}^i(0)$ and using periodicity of $f_{c,d}$.\\
\begin{theorem}
	Let $f_{c,d}(x)=x^d+c \in \bar{\mathbb{Q}}[x]$ be a PCF polynomial having exact type $(m,n)$ with $m\geq 1$. Set $K=\mathbb{Q}(c)$, and let $O_{f_{c,d}} = \{a_1, a_2,\dots , a_{m+n-1}\} \subset \mathcal{O}_K$ be the critical orbit of $f_{c,d}$. Then the following holds:\\
	
	\item[(a)] If $n \not|$ $i$, then $a_i$ is a unit.
	\item[(b)] If $d$ is a prime and $ n \text{ }| \text{ }i$, then one has $(a_i)^A = (d)$, where
	$$
	A = \left\{
	\begin{array}{ll}
	d^{m-1}(d-1) & \text{if } n  \not|  \text{ }m-1 \\
	(d^{m-1}-1)(d-1) & \text{if } n  \text{ }|  \text{ }m-1.
	\end{array}
	\right.
	$$
\end{theorem}

Having stated Theorem $1.3$, two remarks are in order here:\\

 Firstly, taking $i=1$ in Theorem $1.3$, it follows that $a_1 = c$ is always a unit unless $n=1$, which is what is proven in (\cite{Buff}, Proposition $2$). Hence, our theorem generalizes this result of Buff.\\

 Secondly, our proof of the part $n |$$i$ of Theorem $1.3$ only works when $d$ is a prime. In fact, it is easy to come up with counterexamples for this part of the statement when $d$ is not a prime. For example, taking $(d,m,n) = (4,3,1)$, MAGMA gives that $(a_i)^{60} = (4)$ for all $i$, although Theorem $1.3$ would imply $(a_i)^{45} = (4)$. However, when $d$ is a prime power, based on MAGMA computations, perhaps interestingly, it appears that some power of $(a_i)$ gives the ideal $(d)$ for all $i$ divisible by $n$. See Appendix for some more details about these computations. The question of whether for all prime powers $d$ such a power exists or not remains open.

In our next theorem, we are able to get rid of the condition that $d$ is a prime. However, it comes with the price that we do not get as much information as in Theorem $1.3$.

\begin{theorem}
	Let $f_{c,d}(x) = x^d+c \in \bar{\mathbb{Q}}[x]$ be a PCF polynomial with exact type $(m,n)$. Suppose $m\neq 0$. Set $K=\mathbb{Q}(c)$, and let $O_{f_{c,d}} = \{a_1, a_2,\dots , a_{m+n-1}\} \subset \mathcal{O}_K$ be the critical orbit of $f_{c,d}$. Then, $(a_i) | (d)$ for all $1\leq i \leq m+n-1$.
\end{theorem}

Theorem $1.4$ has an application to the Galois theory of polynomial iterates, which we state as our next corollary:
\begin{corollary}
	Let $f_{c,d}(x) = x^d+c \in \bar{\mathbb{Q}}[x]$ be a PCF polynomial with exact type $(m,n)$. Suppose $m\neq 0$. Set $K=\mathbb{Q}(c)$, let $K_n$ be the splitting field of $f_{c,d}^n(x)$ over $K$. If a prime $\mathfrak{p}$ of $K$ ramifies in $K_n$, then $\mathfrak{p} | $$(d)$.
\end{corollary}

\section{Proofs of Main Results}

We first start by stating the following lemma, which we will use throughout the paper. Although it is simple, it becomes suprisingly useful.
\begin{lemma}
	Let $f_{c,d}(x) = x^d + c \in \bar{\mathbb{Q}}[x]$ be a PCF polynomial with exact type $(m,n)$. Set $K=\mathbb{Q}(c)$, and let $O_{f_{c,d}} =   \{a_1,a_2,\dots,a_{m+n-1}\}\subset \mathcal{O}_K$ be the critical orbit of $f$. Then for any $i,j\geq 1$, there exists a polynomial $P_{i,j}(t)\in \mathbb{Z}[t]$ such that $f_{c,d}^i(a_j) = a_j^{d^i}+{a_j}^dP_{i,j}(c)+a_k$, where $k$ is the integer satisfying $1\leq k \leq j$ and $k\equiv i $$($mod $ j)$. Moreover, if $d$ is a prime, we have $\frac{P_{i,j}(t)}{d}\in \mathbb{Z}[t]$.
\end{lemma}
\begin{proof}
	Note that the consant term of $f_{c,d}^i(t)\in \mathbb{Z}[c][t]$ is $a_i$, and all the other terms are divisible by $t^d$. Write \begin{equation}
		f_{c,d}^i(t) = t^{d^i} + t^dF(t) + a_i
	\end{equation} for some $F(t)\in \mathbb{Z}[c][t]$. Plugging $a_j$ into $(2.1)$, get
	\begin{equation}
		f_{c,d}^i(a_j) = a_j^{d^i} + a_j^dF(a_j) + a_i.
	\end{equation}
	If $i\leq j$, by taking $P_{i,j}(c) = F(a_j)$, $(2.2)$ already proves the statement (recall that $a_j\in \mathbb{Z}[c]$). Suppose $i > j$. Let $i=jr+k$ for $r\geq 1$ and $1\leq k \leq j$. By the definitions of $a_i$ and $a_j$, we have \begin{equation}
		a_i = f_{c,d}^{i-j}(a_j) = a_j^dG(c)+a_k
	\end{equation} for some $G(t)\in \mathbb{Z}[c][t]$. Combining this with $(2.2)$, the first part of the statement directly follows. To prove the last statement: Note that by the definition of $f_{c,d}$ and using binomial expansion repeatedly, some ${d\choose l}$ for $1\leq l\leq d-1$ will appear in each coefficient of $P_{i,j}(t)$, which proves the result, since $d | {d\choose l}$ when $d$ is a prime.
\end{proof}
Next, we state another lemma which will be one of the ingredients in the proof of Theorem $1.3$. We first need to recall rigid divisibility sequences:

\begin{definition}
	(\cite{Hamblen}) Let $A = \{a_i\}_{i\geq 1}$ be a sequence in a field $K$. We say $A$ is a rigid divisibility sequence over $K$ if for each non-archimedean absolute value $|$ $.$ $|$ on $K$, the following hold:
	\item[(1)] If $|a_n|<1$, then $|a_n|=|a_{kn}|$ for any $k\geq 1$.
	\item[(2)] If $|a_n|<1$ and $|a_j|<1$, then $|a_{gcd(n,j)}|<1$.
\end{definition}
\begin{remark}
	Let $K$ be a number field, and $\mathcal{O}_K$ be its ring of integers. Suppose $\{a_i\}_{i\geq 1}\subset \mathcal{O}_K$ is a rigid divisibility sequence. Then, the following is a straightforward consequence of Definition $2.2$:
	\item[(i)] $(a_i)$$|$$(a_{ki})$ for all $k\geq 1$.
	\item[(ii)] $((a_i),(a_j)) = (a_{(i,j)})$.
\end{remark}

\begin{lemma}
	Let $f_{c,d}(x) = x^d + c \in \bar{\mathbb{Q}}[x]$ be a PCF polynomial with exact type $(m,n)$. Set $K=\mathbb{Q}(c)$, and let $O_{f_{c,d}} =   \{a_1,a_2,\dots,a_{m+n-1}\}\subset \mathcal{O}_K$ be the critical orbit of $f_{c,d}$. Then, $(a_i) = (a_j)$ for all $i,j$ with $(i,n) = (j,n)$.
\end{lemma}

\begin{proof}
	First note that the sequence $\{a_i\}$ is a rigid divisibility sequence (For a proof of this fact, see (\cite{Hamblen}, Lemma $12$)). We will now prove the lemma by showing that $(a_i) = (a_{(i, n)})$ for all $i$. Since the period is $n$, we can choose large enough integer $k$ such that $a_{i+nik} = a_{i+n(ik+1)}$. Using the second part of Remark $2.3$, we have 
	\begin{equation}
		((a_{i+nik}),(a_{i+n(ik+1)})) = (a_{(i+nik, i+n(ik+1))}) = (a_{(i,n)}).
	\end{equation} Hence, we get
	\begin{equation}
		(a_{i+nik}) = (a_{i+n(ik+1)}) = (a_{(i,n)}).
	\end{equation} Using the first part of Remark $2.3$, the equalities in $(2.5)$ give $$(a_{i+nik})=(a_{(i,n)}) | (a_i),$$ since $(i,n) | i$. On the other hand, we also have $i |$ $i+nik$, thus $(a_i) |$$(a_{i+nik}) = (a_{(i,n)})$. Combining these two, we get $(a_i) = (a_{(i,n)})$, which finishes the proof.
\end{proof}
We can finally prove Theorem $1.3$:
\begin{proof}[Proof of Theorem 1.3]
	We will prove $(1)$ and $(2)$ simultaneously. First suppose $n \not |$ $i$. Using Lemma $2.4$, we can find $m\leq j \leq m+n-1$ such that $(a_i) = (a_j)$. So, it is enough to prove the statement for $a_j$. Since the exact type is $(m,n)$, each $a_k$ for $m\leq k \leq m+n-1$ is a root of the polynomial $\phi(x) = f_{c,d}^n(x) - x \in \mathbb{Z}[c][x]$. The constant term of $\phi(x)$ is $a_n$. There exists a polynomial $P(x) \in \mathbb{Z}[c][x]$ satisfying
	\begin{equation}
		\phi(x) = (\prod_{k=m}^{m+n-1} (x-a_i))P(x).
	\end{equation} So, in particular $P(0)\in \mathbb{Z}[c]$. We also have \begin{equation}
		((-1)^n\prod_{i=m}^{m+n-1} a_i)P(0) = a_n.
	\end{equation} Note that there exists a unique $m\leq k \leq m+n-1$ such that $n |$ $k$, and applying Lemma $2.4$ to this $k$ we have $(a_k) = (a_n)$. Hence, dividing both sides of the last equation by $a_k$, right-hand side becomes a unit, which implies $a_j$ is a unit (since it appears on the left-hand side), thus $a_i$ is a unit.\\ 
	
	Now, suppose $n |$$i$. Using Lemma $2.4$, there exists an integer $k$ such that $m\leq nk \leq m+n-1$ and $(a_i) = (a_{nk})$. So, it is enough to prove the statement for $(a_{nk})$. Note that since $f$ has exact type $(m,n)$, we have \begin{equation}
		f_{c,d}^{m+nk}(0) = f_{c,d}^m(0).
	\end{equation} We also have $f_{c,d}^{m+nk}(0) = f_{c,d}^m(a_{nk})$, so we obtain \begin{equation}
		f_{c,d}^m(a_{nk}) = f_{c,d}^m(0).
	\end{equation}
	\begin{equation}
		\iff [f_{c,d}^{m-1}(a_{nk})]^d+c = [f_{c,d}^{m-1}(0)]^d+c.
	\end{equation}
	\begin{equation}
		\iff (f_{c,d}^{m-1}(a_{nk}) - f_{c,d}^{m-1}(0))(\sum_{i=0}^{d-1} [f_{c,d}^{m-1}(a_{nk})]^i[f_{c,d}^{m-1}(0)]^{d-1-i}) = 0.
	\end{equation} Because $f_{c,d}$ has exact type $(m,n)$, we get
	\begin{equation}
		\sum_{i=0}^{d-1} [f_{c,d}^{m-1}(a_{nk})]^i[f_{c,d}^{m-1}(0)]^{d-1-i} = 0.
	\end{equation}
	Using Lemma $2.1$, since $d$ is a prime, we can find a polynomial $P(t)\in \mathbb{Z}[t]$ satisfying \begin{equation}
		f_{c,d}^{m-1}(a_{nk}) = a_{nk}^{d^{m-1}}+da_{nk}^dP(c)+a_{m-1}.
	\end{equation} Putting this into $(2.12)$, we get
	\begin{equation}
		\sum_{i=0}^{d-1} (a_{nk}^{d^{m-1}}+da_{nk}^dP(c)+a_{m-1})^ia_{m-1}^{d-1-i} = 0.
	\end{equation}
	Using $(2.14)$, we can find a polynomial $Q(t)\in \mathbb{Z}[t]$ such that \begin{equation}
		\sum_{i=0}^{d-1} (a_{nk}^{d^{m-1}}+a_{m-1})^ia_{m-1}^{d-1-i}+da_{nk}^dQ(c) = 0.
	\end{equation}
	\begin{equation}
		\iff \sum_{i=0}^{d-1}(\sum_{j=0}^{i}{i\choose j}a_{nk}^{d^{m-1}j}a_{m-1}^{i-j})a_{m-1}^{d-1-i} + da_{nk}^dQ(c) = 0.
	\end{equation}
	\begin{equation}
		\iff \sum_{j=0}^{d-1} (\sum_{i=j}^{d-1} {i\choose
			j})a_{nk}^{d^{m-1}j}a_{m-1}^{d-1-j} + da_{nk}^dQ(c) = 0.
	\end{equation}
	Using the hockey-stick identity, $(2.17)$ becomes
	\begin{equation}
		\sum_{j=0}^{d-1} {d\choose j+1}a_{nk}^{d^{m-1}j}a_{m-1}^{d-1-j} + da_{nk}^dQ(c) = 0.
	\end{equation}
	Observe that all terms of $(2.18)$ except $da_{m-1}^{d-1}$ and $a_{nk}^{d^{m-1}(d-1)}$ are divisible by $da_{nk}^d$ because $d$ is a rational prime. Therefore we get
	\begin{equation}
	da_{m-1}^{d-1}+a_{nk}^{d^{m-1}(d-1)} \equiv (\text{mod }da_{nk}^d),
	\end{equation}
	and
	\begin{equation}
	a_{nk}^{d^{m-1}(d-1)} = d(-a_{m-1}^{d-1} + \alpha a_{nk}^d)
	\end{equation}
	with some $\alpha \in \mathcal{O}_K$.\\
	
	If $n\text{ }| \text{ }m-1$, then $a_{nk} = \varepsilon a_{m-1}$ with a unit $\varepsilon$, hence $(2.20)$ takes the form
	\begin{equation}
	(\varepsilon a_{m-1})^{d^{m-1}(d-1)} = d(-a_{m-1}^{d-1} + \alpha (\varepsilon a_{m-1})^d),
	\end{equation}
	and after dividing by $a_{m-1}^{d-1}$ we get
	\begin{equation}
	\varepsilon^{d^{m-1}(d-1)}a_{m-1}^{(d^{m-1}-1)(d-1)} = d(-1+\alpha(\varepsilon^d a_{m-1}))
	\end{equation}
and since the ideals $(-1+\alpha\varepsilon^da_{m-1})$ and $(a_{m-1})$ are co-prime we get the assertion.\\

If $n \not | \text{ }m-1$, then $a_{m-1}$ is a unit, hence the ideals $(a_{nk})$ and $(-a_{m-1}^{d-1}+\alpha a_{nk}^d)$ are co-prime and the assertion follows.
\end{proof}

We will now prove the Corollary $1.1$ and Corollary $1.2$. We first need to recall a basic fact from algebraic number theory:\\

Let $L$ be a finite extension of a number field $K$. Let $\mathfrak{p}$ be a prime ideal in $K$. Suppose that $\mathfrak{p}$ factors in $L$ as
$$\mathfrak{p}\mathcal{O}_L = \mathfrak{P_1}^{e_1}\dots\mathfrak{P_g}^{e_g}.$$
Set $f_i = |\mathcal{O}_L/\mathfrak{P_i}|$ for $1\leq i \leq g$, and $n=[L:K]$. Then, we have \begin{equation}
	\sum_{i=1}^{g} e_if_i = n.
\end{equation}  
\begin{proof}[Proof of Corollary 1.1]
	\item[(1)] Set $K=\mathbb{Q}(c)$, and let $N=[K:\mathbb{Q}]$. For $n=1$, by (\cite{Hutz}, Corollary $3.3$), $G_d(m,n)$ has degree $(d^{m-1}-1)(d-1)$. Thus, we have $N\leq (d^{m-1}-1)(d-1)$. On the other hand, second part of Theorem $1.3$ gives $(a_i)^{(d^{m-1}-1)(d-1)} = (d)$ for any $i$. Factor $(a_i)$ into prime factors as
	$$(a_i) = \mathfrak{P_1}^{e_1}\dots \mathfrak{P_g}^{e_g}.$$ Taking the $(d^m-1)(d-1)$th power of each side, we get
	$$ (d) = \mathfrak{P_1}^{(d^m-1-1)(d-1)e_1}\dots \mathfrak{P_g}^{(d^m-1-1)(d-1)e_g}.$$ Using $(2.23)$, we get $$\sum_{i=1}^{g} (d^{m-1}-1)(d-1)e_if_i = N, $$ which implies $N\geq (d^{m-1}-1)(d-1)$. Hence, we obtain $N=(d^{m-1}-1)(d-1)$, which shows $G_d(m,n)$ is irreducible.
	\item[(2)] Let $N$ be as in the first part of the proof. For $n=2$, $d=2$, by (\cite{Hutz}, Corollary $3.3$), $G_d(m,n)$ has degree $2^{m-1}-1$ if $2| $ $m-1$, and has degree $2^{m-1}$ if $2\not |$ $m-1$. On the other hand, using the first part of Theorem $1.3$, we have $(a_i)^{2^{m-1}-1} = (2)$ for any $2| $ $i$ if 
	$m$ is odd, and $(a_i)^{2^{m-1}} = (2)$ for any $2 |$ $i$ if $m$ is even. Similar to the first part, for both cases N becomes equal to the degree of $G_d(m,n)$, which proves that $G_d(m,n)$ is irreducible.
\end{proof}
\begin{remark}
	In $(1)$, $d$ is totally ramified in $K$, and for all $i$, $(a_i)$ is the unique prime ideal of $\mathcal{O}_K$ that divides $(d)$. In $(2)$, $2$ totally ramified in $K$, and for all $i$ even, $(a_i)$ is the unique prime ideal of $\mathcal{O}_K$ that divides $(2)$.
\end{remark}
\begin{proof}[Proof of Corollary 1.2]
	By (\cite{Hamblen}, Theorem $8$), it suffices to show that there is no $\pm d$th power in the orbit. By the first part of Remark $2.5$, $a_i$ is a prime element of $\mathcal{O}_K$ for all $i$. Hence, $a_i$ can never be $\pm d$th power in $\mathcal{O}_K$ for $d\geq 2$, which finishes the proof.
\end{proof}

\begin{proof}[Proof of Theorem 1.4]
	From Theorem $1.3$, we already have that $a_i$ is a unit when $n\not$ $|$ $i$, which gives $(a_i) | (d)$. So, we only need to show $(a_i)|(d)$ for $n|i$. There exists a unique $k$ such that $m\leq nk\leq m+n-1$, i.e., $a_{nk}$ is periodic under $f$. By Lemma $2.4$, it is enough to prove the statement for $a_{nk}$. This is the same situation as in the second part of the proof of Theorem $1.3$. We rewrite the equation $(2.12)$:
	$$\sum_{i=0}^{d-1} [f_{c,d}^{m-1}(a_{nk})]^i[f_{c,d}^{m-1}(0)]^{d-1-i} = 0.$$ 
	Recalling that $d$ is not necessarily a prime, and making the obvious modifications in the equations $(2.13)-(2.19)$ accordingly, we obtain
	\begin{equation}
	da_{m-1}^{d-1}+a_{nk}^{d^{m-1}(d-1)} = \alpha a_{nk}^d
	\end{equation}
	with some $\alpha\in \mathcal{O}_K$.\\
	
	If $n\text{ }| \text{ }m-1$, then $a_{nk} = \varepsilon a_{m-1}$ with a unit $\varepsilon$, hence $(2.24)$ takes the form
	\begin{equation}
	da_{m-1}^{d-1}+(\varepsilon a_{m-1})^{d^{m-1}(d-1)} = \alpha (\varepsilon a_{m-1})^d
	\end{equation}
	and after dividing by $a_{m-1}^{d-1}$ we get
	\begin{equation}
	d+\varepsilon^{d^{m-1}(d-1)}a_{m-1}^{(d^{m-1}-1)(d-1)} = \alpha \varepsilon^d a_{m-1},
	\end{equation}
	which proves the assertion.\\
	
	If $n \not | \text{ }m-1$, then $a_{m-1}$ is a unit, hence $(2.24)$ directly proves the assertion.
	
\end{proof}
We finish this section by proving Corollary $1.5$, and giving a remark about it.

\begin{proof}[Proof of Corollary 1.5]
	Recall that if a prime $\mathfrak{p}$ of $K$ is ramified in $K_n$, then it must divide Disc$(f_{c,d}^n)$. Set $\triangle_n =$ Disc$(f_{c,d}^n)$. By (\cite{Jones}, Lemma 2.6), we have the relation \begin{equation}
		\triangle_n =\pm \triangle_{n-1}d^{d^n}a_n.
	\end{equation} Proceeding inductively, $(2.27)$ shows that if a prime $\mathfrak{p}$ of $K$ ramifies in $K_n$, then $\mathfrak{p}$ divides $(d)$ or $(a_i)$ for some $i$. By Theorem $1.4$, this directly implies $\mathfrak{p} |(d)$, as desired.
\end{proof}
\begin{remark}
	Let $f_{c,2}(x)=x^2+c$ have exact type $(m,1)$ or $(m,2)$, and set $K=\mathbb{Q}(c)$. $2$ is totally ramified in $K$ (by Remark $2.5$), let $\mathfrak{p}\subset \mathcal{O}_K$ be the unique prime above it (which is generated by one of the critical orbit elements). Then, it follows from Corollary $1.5$ that $K_n$ is unramified outside of the set $\{\mathfrak{p},\infty\}$ for all $n\geq 1$. Hence, this way we can get infinitely many explicit examples of pro-$2$ extensions of various number fields unramified outside of a finite prime and infinity.  
\end{remark}
\section{Periodic case}
In this section, we will give a simple observation about the case $m=0$.
\begin{lemma}
	Let $f_{c,d}(x) = x^d+c \in \bar{\mathbb{Q}}[x]$ be a PCF polynomial with exact type $(0,n)$. Set $K=\mathbb{Q}(c)$, and let $O_{f_{c,d}} = \{a_1, a_2,\dots , a_n=0\} \subset \mathcal{O}_K$ be the critical orbit of $f_{c,d}$. Then, $a_i$ is a unit in $\mathcal{O}_K$ for all $1\leq i \leq n-1$.
\end{lemma}
\begin{proof}
	In view of $a_{i+1} = a_i^d+c$, $a_{j+1} = a_j^d+c$ one gets after subtraction and multiplication the equality
	$$\prod_{i\neq j} \frac{a_i^d-a_j^d}{a_i-a_j} = 1,$$
	hence all elements
	$$\frac{a_i^d-a_j^d}{a_i-a_j} \text{  }\text{ }\text{ } (i\neq j)$$
	are units, and putting here $j = n$ we obtain that for $i\neq j$ $a_i^{d-1}$ is a unit, and so is $a_i$.
\end{proof}
\section{Appendix}
We finish the paper by presenting some data about the question stated at the bottom of page $3$. Note that $G_d(m,n)$ is not irreducible in the cases below, and MAGMA computations show that we get different values of $A$ (with the notation of Theorem $1.3$) depending on the minimal polynomial of $c$ over $\mathbb{Q}$.\\

$\bullet (d,m,n) = (4,2,1) \implies (a_i)^{6} = (4)$ or $(a_i)^{12} = (4) $ for all $i$.\\

$\bullet (d,m,n) = (4,2,2) \implies (a_i)^{8} = (4)$ or $(a_i)^{16} = (4) $ for all $2|i$.\\

$\bullet (d,m,n) = (4,2,3) \implies (a_i)^{8} = (4)$ or $(a_i)^{16} = (4) $ for all $3|i$.\\

$\bullet (d,m,n) = (4,3,1) \implies (a_i)^{30} = (4)$ or $(a_i)^{60} = (4) $ for all $i$.\\

$\bullet (d,m,n) = (4,3,2) \implies (a_i)^{30} = (4)$ or $(a_i)^{60} = (4) $ for all $2|i$.\\

$\bullet (d,m,n) = (4,4,1) \implies (a_i)^{126} = (4)$ or $(a_i)^{252} = (4) $ for all $i$.\\

$\bullet (d,m,n) = (8,2,1) \implies (a_i)^{21} = (8)$ or $(a_i)^{42} = (8)$ or $(a_i)^{84} = (8)$  for all $i$.\\

$\bullet (d,m,n) = (9,2,1) \implies (a_i)^{32} = (9)$ or $(a_i)^{96} = (9)$ for all $i$.\\

We also would like to note that this phenomenon does not necessarily hold when $d$ is not a prime power. For instance, taking $(d,m,n) = (6,3,1)$, one sees that $(6) = \mathfrak{p_1}^2 \mathfrak{p_2}$ for some prime ideals $\mathfrak{p_1}, \mathfrak{p_2}$ in $\mathbb{Q}(c)$, which shows that $(6)$ cannot be a perfect power.

\subsection*{Acknowledgments}
The author owes Nigel Boston and Sarah Koch a debt of gratitude for their very helpful comments on this work. The author also would like to thank the anonymous referee for their careful reading of the article and helpful comments.

%%%%%%%%%%% To ease editing, use normal size for the references:

\end{document}